\documentclass{amsart}
\usepackage{stmaryrd}
\usepackage{silence}
\usepackage{upgreek}
\usepackage{amssymb}
\usepackage{faktor}
\usepackage[all,cmtip]{xy}
\usepackage{amsmath} 
\usepackage{mathrsfs}
\usepackage{tikz}
\usepackage{tikz-cd}
\usepackage{enumitem}
\usepackage{mathtools}
\usepackage[mathcal]{euscript}
\usepackage{graphicx}
\usepackage{url}
\usepackage{hyperref}
\usepackage[T1]{fontenc}
\usepackage{subfigure}

\usetikzlibrary{shapes.geometric}
\usetikzlibrary{decorations.markings}
\usetikzlibrary{patterns,decorations.pathreplacing}

\usepackage{ragged2e}

\newcommand{\Conf}{\mathrm{Conf}}

\newcommand{\pt}{\mathrm{pt}}

\definecolor{coloryellow}{RGB}{240,228,66}
\definecolor{colorskyblue}{RGB}{86,180,233}
\definecolor{colorvermillion}{RGB}{213,94,0}

\newcommand{\graphfont}{\mathsf}

\newcommand{\thetagraph}[1]{\graphfont{\Theta}_{#1}}

\newcommand{\stargraph}[1]{\graphfont{S}_{#1}}
\newcommand{\graf}{\graphfont{\Gamma}}

\DeclareSymbolFont{sfletters}{OT1}{cmss}{m}{n}
\DeclareMathSymbol{\sTheta}{\mathord}{sfletters}{"02}

\theoremstyle{definition}
\newtheorem{definition}{Definition}[section]

\theoremstyle{plain}
\newtheorem{proposition}[definition]{Proposition}
\newtheorem{lemma}[definition]{Lemma}

\newtheorem{theorem}[definition]{Theorem}

\theoremstyle{remark}
\newtheorem{remark}[definition]{Remark}

\makeatletter
\@addtoreset{definition}{section}
\makeatother

\usepackage{marginnote}
    \DeclareFontFamily{U}{wncy}{}
    \DeclareFontShape{U}{wncy}{m}{n}{<->wncyr10}{}
    \DeclareSymbolFont{mcy}{U}{wncy}{m}{n}
    \DeclareMathSymbol{\Sha}{\mathord}{mcy}{"58}

\newsavebox{\foobox}

\title{The topological complexity of pure graph braid groups is stably maximal}
\author{Ben Knudsen}
\email{b.knudsen@northeastern.edu}
\address{Department of Mathematics, Northeastern University, Boston, MA 02115, USA}

\begin{document}

\begin{abstract}
We prove Farber's conjecture on the stable topological complexity of configuration spaces of graphs. The conjecture follows from a general lower bound derived from recent insights into the topological complexity of aspherical spaces. Our arguments apply equally to higher topological complexity.
\end{abstract}

\maketitle

\section{Introduction}

The difficulty of collision-free motion planning is measured by the topological complexity of configuration spaces. In many situations, the ambient workspace is naturally modeled as a graph, and the corresponding complexity was described conjecturally by Farber in 2005 \cite{Farber:CFMPG}. Our purpose is to prove this conjecture.

Write $\Conf_k(\graf)$ for the configuration space of $k$ ordered points in the graph $\graf$, denote the number of vertices of $\graf$ of valence at least $3$ by $m(\graf)$, and write $\mathrm{TC}_r$ for the $r$th higher topological complexity \cite{Farber:TCMP, Rudyak:HATC}. Farber's conjecture is the case $r=2$ of the following result.

\begin{theorem}\label{thm:farber}
Let $\graf$ be a connected graph with $m(\graf)\geq2$. For $k\geq 2m(\graf)$ and $r>0$, we have the equality 
\[\frac{1}{r}\,\mathrm{TC}_r(\Conf_k(\graf))=m(\graf).\]
\end{theorem}

The case $m(\graf)=0$ is trivial, and the case $m(\graf)=1$ is completely understood---see \cite[Thm. A]{LuetgehetmannRecio-Mitter:TCCSFAGBG}, for example.

Using a cohomological argument, Farber proved this result for trees and $r=2$ \cite{Farber:CFMPG}. This argument was later adapted to establish the same result for fully articulated and banana graphs and $r=2$ \cite{LuetgehetmannRecio-Mitter:TCCSFAGBG}, for trees and $r>0$ \cite{Aguilar-GuzmanGonzalezHoekstra-Mendoza:FSMTMHTCOCST}, and for planar graphs and $r>0$ \cite{Knudsen:FCPG}. 

The idea of Farber's argument is easily explained. First, the quantity $\frac{1}{r}\mathrm{TC}_r$ is bounded above by the homotopy dimension of the background space, which for $\Conf_k(\graf)$ is $m(\graf)$ in the regime of interest. Second, this dimension is detected by a map from a torus of dimension $m(\graf)$ (described geometrically below). Third, one can produce long cup products in cohomology by combining classes detecting circle factors of this torus---provided such classes exist. Since a variant of the cup length bounds topological complexity from below, one obtains the desired conclusion.

Unfortunately, the cohomological decomposability of the torus in question is equivalent to planarity \cite{Knudsen:FCPG}, so the requisite classes fail to exist in the non-planar setting. We circumvent this obstacle by working at the level of fundamental groups, leveraging a theorem of \cite{FarberOprea:HTCAS}, following \cite{GrantLuptonOprea:NLBTCAS}, on the higher topological complexity of aspherical spaces to establish the necessary lower bound (Theorem \ref{thm:general lower bound}).


\subsection{Conventions} We say that subgroups of a fixed group are disjoint if they intersect trivially. We write $\mathrm{cd}_R(G)$ for the cohomological dimension of the group $G$ over the commutative ring $R$, which is to say the minimal length of a projective resolution of the trivial module $R$ over the group ring $R[G]$. We set $\mathrm{cd}(G)=\mathrm{cd}_\mathbb{Z}(G)$.

Given a path $\gamma$ in a space $X$, we denote its path homotopy class by $[\gamma]$, its reverse by $\bar\gamma$, and its change-of-basepoint isomorphism by $\hat\gamma:\pi_1(X,\gamma(0))\to \pi_1(X,\gamma(1))$. We denote the concatenation operation on paths by $\star$.

We use the ``reduced'' convention for topological complexity, in which $\mathrm{TC}(\pt)=0$.

\subsection*{Acknowledgements} The seeds of this paper were sown when Andrea Bianchi noticed an early error in \cite{Knudsen:FCPG}, prompting the author to reconsider the role of cohomology in the non-planar setting. The ideas came during the 2022 workshop ``Topological Complexity and Motion Planning,'' held at CMO BIRS in Oaxaca, and the author thanks Dan Cohen, Jes\'{u}s Gonz\'{a}lez, and Lucile Vandembroucq for their skillful organization. The author benefited from fruit-full conversations with Daciberg Gon\c{c}alves and Teresa Hoekstra Mendoza, and he thanks Andrea Bianchi, Jes\'{u}s Gonz\'{a}lez, and the anonymous referee for comments on an earlier draft. This work was supported by NSF grant DMS-1906174.

\section{Graphs and their configuration spaces}

For our purposes, a graph is a finite CW complex $\graf$ of dimension at most $1$. The degree or valence $d(v)$ of a vertex $v$ of $\graf$ is the number of connected components of its complement in a sufficiently small open neighborhood. A vertex of valence at least $3$ is called essential.

The configuration spaces of a graph $\graf$ depend only on the homeomorphism type of $\graf$; in particular, they are invariant under subdivision. The fundamental fact concerning these configuration spaces is the following.

\begin{theorem}[{\cite{Abrams:CSBGG}}]
For any graph $\graf$ and $k\geq0$, the space $\Conf_k(\graf)$ is aspherical.
\end{theorem}

In particular, if $\graf$ is connected and $m(\graf)>0$, then the configuration space $\Conf_k(\graf)$ is a classifying space for its fundamental group, the pure graph braid group on $k$ strands, denoted $P_k(\graf)$. We work with a basepoint $x_0\in\Conf_k(\graf)$ chosen so that every coordinate is a bivalent vertex, which is always achievable after subdivision.

Fixing a vertex $v$, consider the subspace $\iota_v:\Conf_k(\graf)_v\subseteq \Conf_k(\graf)$ consisting of configurations with at most one coordinate in the open star of $v$.

\begin{proposition}\label{prop:local configurations}
If the open star of $v$ is contractible, then $\iota_v$ is a homotopy equivalence.
\end{proposition}
\begin{proof}
The assumption guarantees that $\Conf_k(\graf)_{v}$ concides with the subspace considered in \cite[Lem. 2.0.1]{AgarwalBanksGadishMiyata:DCCSG} after subdividing.
\end{proof}

In the remainder of the paper, we work with a connected graph $\graf$. We subdivide $\graf$ to guarantee that every essential vertex has a contractible open star (equivalent to requiring that $\graf$ have no self-loops) and that distinct essential vertices have disjoint closed stars. We fix a parametrization of each edge of $\graf$ and an ordering of the set of edges at each essential vertex.

\section{A general lower bound}\label{section:argument}

The ``higher'' or ``sequential'' topological complexity $\mathrm{TC}_r$ is a numerical homotopy invariant defined in \cite{Rudyak:HATC} and developed in \cite{BasabeGonzalezRudyakTamaki:HTCS}, recovering Farber's topological complexity \cite{Farber:TCMP} for $r=2$ and the Lusternik--Schnirelmann category for $r=1$. The reader is referred to these references for definitions; we recall only that $\frac{1}{r}\mathrm{TC}_r$ is bounded above by the homotopy dimension in non-pathological settings, e.g., for spaces homotopy equivalent to CW complexes.

Theorem \ref{thm:farber} is an immediate consequence of the following result, together with the dimensional bound and the well-known fact that the homotopy dimension of $\Conf_k(\graf)$ is bounded above by $m(\graf)$, independent of $k$ \cite{Swiatkowski:EHDCSG}.

\begin{theorem}\label{thm:general lower bound}
Let $\graf$ be a connected graph. For any $r> 0$ and $k\geq4$, we have the inequality 
\[
\frac{1}{r}\,\mathrm{TC}_r(\Conf_k(\graf))\geq \min\left\{\left\lfloor\frac{k}{2}\right\rfloor, m(\graf)\right\}.
\]
\end{theorem}

This result recovers theorems of \cite{Aguilar-GuzmanGonzalezHoekstra-Mendoza:FSMTMHTCOCST} and \cite{Knudsen:FCPG} for trees and planar graphs, respectively. In these cases, the estimate follows from consideration of the zero-divisor cup length in the cohomology of the relevant configuration space. If such a strategy is available in the non-planar setting, it is not without significant and non-obvious modification \cite[Thm. 8.1]{Knudsen:FCPG}. Instead, we aim to leverage the advances in the study of the topological complexity of aspherical spaces made in the wake of \cite{GrantLuptonOprea:NLBTCAS}.

\begin{theorem}[{\cite[Thm. 2.1]{FarberOprea:HTCAS}}]\label{thm:lower bound}
Let $G$ be a discrete group with subgroups $H$ and $K$, and fix $r\geq 2$. If every conjugate of $H$ is disjoint from $K$, then \[\mathrm{TC}_r(BG)\geq \mathrm{cd}(H\times K\times G^{r-2}).\]
\end{theorem}

Our search for suitable subgroups will be organized by the following simple combinatorial object.

\begin{definition}
Fix a ground set $S$, a graph $\graf$, and a set $W$ of essential vertices. A \emph{binary $W$-partition} of $S$ is a collection $\lambda=\{\lambda(v)\}_{v\in W}$ of disjoint subsets of $S$ of cardinality $2$ whose union is $S$. We say that binary $W$-partitions $\lambda$ and $\mu$ are \emph{orthogonal}, and write $\lambda\perp\mu$, if $\lambda(v)\neq\mu(w)$ for every $(v,w)\in W\times W$.
\end{definition}

We will apply Theorem \ref{thm:lower bound} to subgroups $T_\lambda:=\mathrm{im}(\tau_\lambda)$, where $\lambda$ is a binary $W$-partition of $\{1,\ldots, k\}$ and $\tau_\lambda:\mathbb{Z}^W\to P_k(\graf)$ a certain homomorphism. To study these toric subgroups, we construct detection homomorphisms $\delta_\lambda:P_k(\graf)\to G_W$, where $G_W$ is a certain product of free groups. These homomorphisms are constructed in Section \ref{section:tori}, and they interact according to the following result, whose proof is taken up in Section \ref{section:proofs}.

\begin{proposition}\label{prop:homomorphisms}
Let $\lambda$ and $\mu$ be binary $W$-partitions of $\{1,\ldots, k\}$.
\begin{enumerate}
\item If $\lambda=\mu$, then the composite $\mathbb{Z}^{W}\xrightarrow{\tau_\lambda} P_k(\graf)\xrightarrow{\delta_\mu} G_W$ is injective.
\item If $\lambda\perp \mu$, then the composite $\mathbb{Z}^{W}\xrightarrow{\tau_{\lambda}} P_k(\graf)\xrightarrow{\delta_{\mu}} G_W$ is trivial.
\end{enumerate}
\end{proposition}

Assuming this result, we prove the theorem.

\begin{proof}[Proof of Theorem \ref{thm:general lower bound}] The cases $m(\graf)\in\{0,1\}$ are easily treated by other means, so assume otherwise. By \cite[Prop. 5.6]{LuetgehetmannRecio-Mitter:TCCSFAGBG} and \cite[p. 4]{Farber:TCMP}, the quantity $\mathrm{TC}_r(\Conf_{k}(\graf))$ is non-decreasing in $k$ for fixed $r$. Thus, we may assume without loss of generality that $k=2d$ with $2\leq d\leq m(\graf)$. These assumptions imply that there exist binary $W$-partitions $\lambda$ and $\mu$ of $\{1,\ldots, 2d\}$ with $\lambda\perp\mu$ for some set $W$ of essential vertices. Our aim is to show that $\mathrm{TC}_r\geq rd$.

Assume that $r\geq2$. According to Proposition \ref{prop:homomorphisms}, we have the containment $T_{\lambda}\subseteq \ker(\delta_{\mu})$; therefore, since the latter subgroup is normal, it contains every conjugate of the former. Since $\ker(\delta_{\mu})$ is disjoint from $T_{\mu}$ by the same result, the group $G=P_k(\graf)$ and subgroups $H=T_\lambda$ and $K=T_\mu$ satisfy the hypotheses of Theorem \ref{thm:lower bound}. We conclude the relations 
\begin{align*}
\mathrm{TC}_r(\Conf_{2d}(\graf))&=\mathrm{TC}_r(BP_{2d}(\graf))\\
&\geq\mathrm{cd}(T_\lambda\times T_\mu\times P_{2d}(\graf)^{r-2})\\
&=\mathrm{cd}(\mathbb{Z}^{d}\times\mathbb{Z}^d\times P_{2d}(\graf)^{r-2})\\
&\geq\mathrm{cd}_\mathbb{Q}(\mathbb{Z}^{d}\times\mathbb{Z}^d\times P_{2d}(\graf)^{r-2}),
\end{align*} where the first uses homotopy invariance and asphericity, the second is Theorem \ref{thm:lower bound}, the third follows from Proposition \ref{prop:homomorphisms}, and the fourth follows by extension of scalars. Thus, it suffices to show that the rational homology of the group $\mathbb{Z}^{d}\times\mathbb{Z}^d\times P_{2d}(\graf)^{r-2}$ is nonzero in degree $rd$. Every group in sight has homology of finite type, so the K\"{u}nneth isomorphism applies, and the conclusion follows from the fact that $H_d(\Conf_{2d}(\graf);\mathbb{Q})\neq0$ for $2\leq d\leq m(\graf)$, which is well known---see \cite[Lem. 3.18]{AnDrummondColeKnudsen:ESHGBG}, for example.

In the case $r=1$, the claimed bound is well known; indeed, by asphericity and the Eilenberg--Ganea theorem \cite{EilenbergGanea:LSCAG}, the quantity to be estimated is simply $\mathrm{cd}(P_k(\graf))$, and the same rational homology calculation applies.
\end{proof}

\section{Star graphs and theta graphs}\label{section:stars and thetas}

This section is devoted to the simple calculation forming the basis for our later arguments.

\begin{definition}
The \emph{star graph} $\stargraph{n}$ is the cone on the discrete space $\partial\stargraph{n}:=\{v_1,\ldots, v_n\}$. The \emph{theta graph} $\thetagraph{n}$ is the quotient $\stargraph{n}/\partial \stargraph{n}$.
\end{definition}

Each star and theta graph has a canonical graph structure, which is canonically parametrized. We denote the unique essential vertex of $\stargraph{n}$, and its image in $\thetagraph{n}$, by $v_0$. We denote the image of $\partial \stargraph{n}$ in $\thetagraph{n}$ by $v_\infty$. Note that $\thetagraph{n}$ is (non-canonically) homotopy equivalent to a bouquet of $n-1$ circles.

The case $n=3$ is of fundamental importance. Within the configuration space $\Conf_2(\stargraph{3})$, there is the loop given by the sixfold concatenated path \[\epsilon=(\underline v_1, e_{23})\star (e_{12}, \underline v_3)\star (\underline v_2, e_{31})\star (e_{23}, \underline v_1)\star (\underline v_3, e_{12})\star (e_{31}, \underline v_2) ,\] where $\underline v_i$ denotes the constant path at $v_i$ and $e_{ij}:[0,1]\to \stargraph{3}$ the unique piecewise linear path from $v_i$ to $v_j$ with $e_{ij}(\frac{1}{2})=v_0$. Note that $\epsilon$ is a loop based at the configuration $(v_1, v_2)$. 

It is a standard fact that $\epsilon:S^1\to \Conf_2(\stargraph{3})$ is a homotopy equivalence. More precisely, we have the following.

\begin{lemma}\label{lem:local homeomorphism}
The map $\epsilon$ factors through an embedding $\tilde\epsilon: S^1\to \Conf_2(\stargraph{3})_{v_0}$ as a deformation retract.
\end{lemma}
\begin{proof}
The image of $\epsilon$ lies in $\Conf_2(\stargraph{3})_{v_0}$ by construction, so $\tilde \epsilon$ exists. It is easy to check that $\epsilon$ is injective; therefore, since $S^1$ is compact and $\Conf_2(\stargraph{3})$ Hausdorff, it follows that $\tilde \epsilon$ is an embedding. The space $\Conf_2(\stargraph{3})_{v_0}$ is obtained from the image of $\tilde \epsilon$ by attaching half-open intervals, and collapsing these intervals provides a deformation retraction.
\end{proof}

We define a function $q:\Conf_2(\stargraph{3})_{v_0}\to \thetagraph{3}$ by recording the coordinate of any particle in $\stargraph{3}\setminus \partial\stargraph{3}$ and sending all other configurations to $v_\infty$. Using the gluing lemma, one shows easily that $q$ is continuous.

\begin{lemma}\label{lem:star injective}
The following composite homomorphism is injective: \[\pi_1(\Conf_2(\stargraph{3}),(v_1,v_2))\xrightarrow{(\iota_{v_0})_*^{-1}}\pi_1(\Conf_2(\stargraph{3})_{v_0}, (v_1,v_2)) \xrightarrow{q_*}\pi_1(\thetagraph{3}, v_\infty).\]
\end{lemma}
\begin{proof}
By Proposition \ref{prop:local configurations} and Lemma \ref{lem:local homeomorphism}, it suffices to show that $q_*\tilde \epsilon_*$ is injective. Since the target is a free group, and in particular torsion-free, it suffices to show that the path homotopy class $[q\circ\tilde\epsilon]$ is nontrivial. Writing $\gamma_{ij}$ for the loop in $\thetagraph{3}$ induced by $e_{ij}$, and noting the relations $[\gamma_{31}]=[\gamma_{32}\star\gamma_{21}]$ and $\bar\gamma_{ij}=\gamma_{ji}$, we have \begin{align*}
[q\circ\tilde \epsilon]&=[\gamma_{23}\star \gamma_{12}\star \gamma_{31}\star\gamma_{23}\star \gamma_{12}\star \gamma_{31}]\\
&=[\gamma_{23}\star\gamma_{12}\star\gamma_{32}\star\gamma_{21}\star\gamma_{23}\star\gamma_{12}\star\gamma_{32}\star\gamma_{21}]\\
&=[\gamma_{23}\star\gamma_{12}\star\bar\gamma_{23}\star\bar\gamma_{12}\star\gamma_{23}\star\gamma_{12}\star\bar\gamma_{23}\star\bar\gamma_{12}]\\
&=[\gamma_{23}][\gamma_{12}][\gamma_{23}]^{-1}[\gamma_{12}]^{-1}[\gamma_{23}][\gamma_{12}][\gamma_{23}]^{-1}[\gamma_{12}]^{-1}.
\end{align*} This expression is a nonempty reduced word in the set $\{[\gamma_{12}],[\gamma_{23}]\}$, which forms a system of free generators for $\pi_1(\thetagraph{3},v_\infty)$, implying the claim.
\end{proof}

\begin{remark}
One can interpret the theta graph appearing here (up to homotopy) as the configuration space of two unordered points in the quotient $\stargraph{3}/\partial\stargraph{3}$, regarded as a graph with a ``sink'' vertex at which points are permitted to collide \cite{ChettihLuetgehetmann:HCSTL}. Informally, a deformation retraction is provided by allowing particles to flow down the edges of the graph until at least one reaches the sink. The use of graphs with sinks in \cite{LuetgehetmannRecio-Mitter:TCCSFAGBG} inspired some of the ideas in the present article.
\end{remark}

\begin{remark}
The same calculation shows that the composite of Lemma \ref{lem:star injective} is trivial at the level of homology. Ultimately, it is this fact that forms the obstacle to a cohomological proof of Farber's conjecture for non-planar graphs---see \cite{Knudsen:FCPG} for further discussion.
\end{remark}

\section{Toric and detection homomorphisms}\label{section:tori}

Fix a set $W$ of essential vertices and a binary $W$-partition $\lambda$ of $\{1,\ldots, k\}$. For each $v\in W$, given our choice of parametrization and subdivision, there is a piecewise linear cellular embedding $f_v:\stargraph{3}\to \graf$ uniquely specified by requiring that it send the $i$th edge of $\stargraph{3}$ to the $i$th edge at $v$ for $1\leq i\leq 3$. These embeddings induce an embedding $f$ from a $W$-indexed disjoint union of star graphs (we use our assumption on the closed stars at essential vertices). Consider the composite embedding
\[\varphi_\lambda: (S^1)^W\to \Conf_2(\stargraph{3})^W\to \Conf_{k}\left(\coprod_{v\in W} \stargraph{3}\right)\to\Conf_{k}(\graf),\] where the first map is the canonical map $\epsilon$ in each factor; the second is the inclusion of the subspace in which the coordinates indexed by the elements of $\lambda(v)$ lie in the component indexed by $v$; and the third is induced by the embedding $f$. 

Finally, we choose a path $\alpha_\lambda$ in $\Conf_k(\graf)$ from $\varphi_\lambda(1)$ to $x_0$. We require this path to be a concatenation of coordinatewise linear edge paths in $\graf$ with all but one coordinate stationary. Such a path exists by our assumption that $\graf$ is connected and $m(\graf)>0$ (in particular, $\Conf_k(\graf)$ is connected).

\begin{remark}\label{remark:local factorization}
By construction, the image of the embedding $\varphi_\lambda$ lies in $\Conf_k(\graf)_v$ for every essential vertex $v$. The same holds for the path $\alpha_\lambda$. Our notation will not distinguish among the the various resulting paths, even though they have different codomains.
\end{remark}

With these desiderata in hand, we are ready to define the toric homomorphism associated to $\lambda$. The reader is cautioned that this homomorphism depends on many choices, none of which are reflected in our notation. Most of these choices affect the definition only up to conjugation in $P_k(\graf)$.

\begin{definition}
Let $\lambda$ be a binary $W$-partition of $\{1,\ldots, k\}$. The homomorphism $\tau_\lambda$ is the composite \[\mathbb{Z}^W\cong  \pi_1(S^1,1)^W\xrightarrow{(\varphi_\lambda)_*}\pi_1(\Conf_k(\graf), \varphi_\lambda(1))\xrightarrow{\hat \alpha_\lambda}P_k(\graf).\]
\end{definition}

The target of the detection homomorphism $\delta_\lambda$ will be the product of free groups $G_W:=\prod_{v\in W}\pi_1(\thetagraph{d(v)}, v_\infty))$. To define $\delta_\lambda$, we note that our description of the quotient map $q$ given in Section \ref{section:stars and thetas} in fact describes a map $q_v:\Conf_2(\graf)_v\to \thetagraph{d(v)}$ for any essential vertex $v$ in $\graf$. Note that we use our subdivision assumption as well as our chosen ordering of the edges at $v$. 

In the following definition, $\pi$ denotes coordinate projection.

\begin{definition}
Let $\lambda$ be a binary $W$-partition of $\{1,\ldots, k\}$. The homomorphism $\delta_\lambda: P_k(\graf)\to G_W$ is the homomorphism with $v$th component the composite {\small\[P_k(\graf)\xrightarrow{(\iota_v)_*^{-1}} \pi_1(\Conf_k(\graf)_v, x_0)\xrightarrow{(\pi_{\lambda(v)})_*} \pi_1(\Conf_2(\graf)_v, \pi_{\lambda(v)}(x_0))\xrightarrow{(q_v)_*}  \pi_1(\thetagraph{d(v)}, v_\infty).\]}
\end{definition}

Note that, since each coordinate of our basepoint $x_0$ is a bivalent vertex, the composite in question does send $x_0$ to $v_\infty$. We emphasize that, although the target of $\delta_\lambda$ depends only on $W$, the homomorphism itself depends on $\lambda$.

\section{Proof of Proposition \ref{prop:homomorphisms}}\label{section:proofs}

In this section, we consider the interaction between the toric and detection homomorphisms defined above. At various points, we will employ an unnamed map of the form $\Conf_2(\stargraph{3})\to \Conf_k(\graf)_w$, which is obtained by restricting the map $\Conf_2(\stargraph{3})^W\to \Conf_k(\graf)_w$ determined by a binary $W$-partition along the section of the projection onto the $w$th factor determined by respective basepoints.

The key result is the following.

\begin{lemma}\label{lem:triviality}
Let $\lambda$ and $\mu$ be binary $W$-partitions of $\{1,\ldots, k\}$. For $v,w\in W$, the composite homomorphism
\[
\xymatrix{
\mathbb{Z}\ar[r]^-{v\in W}&\mathbb{Z}^W\ar[r]^-{\tau_{\lambda}}&P_k(\graf)\ar[r]^-{\delta_{\mu}}&G_W\ar[r]^-{w\in W} &\pi_1(\thetagraph{d(w)},v_
\infty)
}
\] is trivial unless $v=w$ and $\lambda(v)=\mu(v)$. 
\end{lemma}
\begin{proof}
After defining the paths $\beta=\pi_{\mu(w)}\circ\alpha_\lambda$ and $\gamma=q_w\circ\beta$, we have the commutative diagram of homomorphisms
{\tiny\[\xymatrix{
\pi_1(S^1,1)\ar[r]^-{v\in W}\ar[d]_-{\epsilon_*}&\pi_1(S^1,1)^W\ar[d]_-{(\varphi_{\lambda})_*}\\
\pi_1(\Conf_2(\stargraph{3}), (v_1,v_2))\ar@{-->}[dr]\ar[r]&\pi_1(\Conf_k(\graf),\varphi_{\lambda}(1))\ar[d]_-{(\iota_{w})_*^{-1}}\ar[r]^-{\hat\alpha_{\lambda}}&\pi_1(\Conf_k(\graf),x_0)\ar[d]^-{(\iota_{w})_*^{-1}}\\
&\pi_1(\Conf_k(\graf)_w,\varphi_\lambda(1))\ar[d]_-{(\pi_{\mu(w)})_*}\ar[r]^-{\hat\alpha_\lambda}&\pi_1(\Conf_k(\graf)_w,x_0)\ar[d]^-{(\pi_{\mu(w)})_*}\\
&\pi_1(\Conf_2(\graf)_w, \pi_{\mu(w)}(\varphi_\lambda(1)))\ar[d]_-{(q_w)_*}\ar[r]^-{\hat\beta}&\pi_1(\Conf_2(\graf)_w, \pi_{\mu(w)}(x_0))\ar[d]^-{(q_w)_*} \\
&\pi_1(\thetagraph{d(w)}, v_\infty)\ar[r]^-{\hat\gamma}&\pi_1(\thetagraph{d(w)}, v_\infty)),
}\]}where the dashed filler arises from Remark \ref{remark:local factorization}. The claim is that the total composite is trivial unless $v=w$ and $\lambda(v)=\mu(v)$. If $v\neq w$, then the composite map \[\Conf_2(\stargraph{3})\to \Conf_k(\graf)_w\xrightarrow{\pi_{\mu(w)}} \Conf_2(\graf)_w\xrightarrow{q_w} \thetagraph{d(w)}\] is the constant map to the basepoint; indeed, the image of the embedding $f_v:\stargraph{3}\to \graf$ is disjoint from the open star of $w$, since it is contained in the closed star of $v$. Here we use the assumption that $v\neq w$ as well as our assumption on the subdivision of $\graf$.

Without loss of generality, then, we may take $v=w$. If $\lambda(v)\cap \mu(v)=\varnothing$, then the composite map 
\[\Conf_2(\stargraph{3})\to \Conf_k(\graf)_v\xrightarrow{\pi_{\mu(v)}} \Conf_2(\graf)_v\] is constant with value $\pi_{\mu(v)}(\varphi_\lambda(1))$, and the claim follows as before; thus, we may assume that $\lambda(v)\cap \mu(v)=\{1\}$. In this case, the loop \[S^1\xrightarrow{\epsilon} \Conf_2(\stargraph{3})\to \Conf_k(\graf)_v\xrightarrow{\pi_{\mu(v)}} \Conf_2(\graf)_v\xrightarrow{q_v} \thetagraph{d(v)}\] is the sixfold concatenation $\underline v_\infty\star \gamma_{12}\star \underline v_\infty\star \gamma_{23}\star \underline v_\infty\star \gamma_{31}$, and we calculate that
\begin{align*}[\underline v_\infty\star \gamma_{12}\star \underline v_\infty\star \gamma_{23}\star \underline v_\infty\star \gamma_{31}]&=[\gamma_{12}\star\gamma_{23}\star\gamma_{31}]\\
&=[\gamma_{12}\star\gamma_{23}\star\gamma_{32}\star\gamma_{21}]\\
&=[\gamma_{12}\star\gamma_{23}\star\bar\gamma_{23}\star\bar\gamma_{12}]\\
&=[\gamma_{12}][\gamma_{23}][\gamma_{23}]^{-1}[\gamma_{12}]^{-1}\\
&=1.
\end{align*} Since the composite in question annihilates a generator, it is trivial.
\end{proof}

The desired result is now within easy reach.

\begin{proof}[Proof of Proposition \ref{prop:homomorphisms}]
Lemma \ref{lem:triviality} implies that the composite $\delta_\lambda\circ\tau_\lambda$ splits as a product of homomorphisms indexed by $W$, so the first claim is equivalent to the injectivity of each of these component homomorphisms. The injectivity of the component indexed by $v\in W$ is equivalent to the injectivity of the total composite in the commutative diagram
{\tiny\[\xymatrix{
\pi_1(S^1,1)\ar[r]^-{v\in W}\ar[d]_-{\epsilon_*}&\pi_1(S^1,1)^W\ar[d]_-{(\varphi_\lambda)_*}\\
\pi_1(\Conf_2(\stargraph{3}), (v_1,v_2))\ar[d]_-{(\iota_{v_0})_*^{-1}}\ar[r]&\pi_1(\Conf_k(\graf),\varphi_\lambda(1))\ar[d]_-{(\iota_{v})_*^{-1}}\ar[r]^-{\hat\alpha_\lambda}&\pi_1(\Conf_k(\graf),x_0)\ar[d]^-{(\iota_{v})_*^{-1}}\\
\pi_1(\Conf_2(\stargraph{3})_{v_0}, (v_1,v_2))\ar@{=}[d]\ar[r]&\pi_1(\Conf_k(\graf)_v,\varphi_\lambda(1))\ar[d]_-{(\pi_{\lambda(v)})_*}\ar[r]^-{\hat\alpha_\lambda}&\pi_1(\Conf_k(\graf)_v,x_0)\ar[d]^-{(\pi_{\lambda(v)})_*}\\
\pi_1(\Conf_2(\stargraph{3})_{v_0}, (v_1,v_2))\ar[d]_-{q_*}\ar[r]&\pi_1(\Conf_2(\graf)_v, \pi_{\lambda(v)}(\varphi_\lambda(1)))\ar[d]_-{(q_v)_*}\ar[r]^-{\hat\beta}&\pi_1(\Conf_2(\graf)_v, \pi_{\lambda(v)}(x_0))\ar[d]^-{(q_v)_*} \\
\pi_1(\thetagraph{3},v_\infty)\ar[r]&\pi_1(\thetagraph{d(v)}, v_\infty)\ar[r]^-{\hat\gamma}&\pi_1(\thetagraph{d(v)}, v_\infty)),
}\]}where the homomorphism in the bottom left is induced by the inclusion $\{1,2,3\}\subseteq\{1,\ldots, d(v)\}$. In terms of systems of free generators, this homomorphism is induced by the inclusion $\{[\gamma_{12}],[\gamma_{23}]\}\subseteq \{[\gamma_{12}],[\gamma_{23}],\ldots, [\gamma_{d(v)-1, d(v)}]\}$; in particular, it is injective. Since $\hat\gamma$ is an isomorphism, the claim follows from Lemma \ref{lem:star injective}, which implies that the composite in the lefthand vertical column is injective.

For the second claim, suppose that $\lambda\perp\mu$. Since $\delta_\mu\circ\tau_\lambda$ is a homomorphism out of a direct sum and into a product, it suffices to establish the triviality of the composite homomorphisms \[
\xymatrix{
\mathbb{Z}\ar[r]^-{v\in W}&\mathbb{Z}^W\ar[r]^-{\tau_\lambda}&P_k(\graf)\ar[r]^-{\delta_\mu}&G_W\ar[r]^-{w\in W}&\pi_1(\thetagraph{d(w)}, v_\infty)
}
\] for every $(v,w)\in W\times W$. By assumption, we have $\lambda(v)\neq \mu(w)$ for every such pair, so the claim follows from Lemma \ref{lem:triviality}.
\end{proof}

%
%

\bibliographystyle{amsalpha}
\bibliography{references}

\end{document}